\documentclass[11pt, leqno]{amsart}%%%%%%%%%%
\usepackage {amssymb}
\usepackage {amsmath}
\usepackage {bbm}
\usepackage{amsthm}
\usepackage{mathtools}
\usepackage{graphicx}
\usepackage {amscd}
\usepackage {epic}
\usepackage[mathscr]{euscript}
\usepackage{comment}

\usepackage{geometry}
\usepackage{bm}
\usepackage[dvipsnames]{xcolor}
\usepackage[colorlinks,citecolor=OliveGreen,linkcolor=Mahogany,urlcolor=Plum,pagebackref]{hyperref}
\usepackage[alphabetic]{amsrefs}

\usepackage{todonotes}

\theoremstyle{plain}
\newtheorem{theorem}{Theorem}%[section]
\newtheorem{lemma}[theorem]{Lemma}

\newtheorem{corollary}[theorem]{Corollary}
\newtheorem{conjecture}[theorem]{Conjecture}

\theoremstyle{definition}
\newtheorem{example}[theorem]{Example}
\newtheorem{definition}[theorem]{Definition}

\theoremstyle{remark}
\newtheorem{remark}[theorem]{Remark}

\newcommand{\cO}{\mathcal{O}}

\newcommand{\bN}{\mathbb{N}}
\newcommand{\bP}{\mathbb{P}}
\newcommand{\bQ}{\mathbb{Q}}

\newcommand{\bR}{\mathbb{R}}

\newcommand{\bZ}{\mathbb{Z}}

\newcommand{\fX}{\mathfrak{X}}

\newcommand{\fa}{\mathfrak{a}}

\newcommand{\fm}{\mathfrak{m}}

\newcommand{\Frac}{\mathrm{Frac}}

\newcommand{\ord}{\mathrm{ord}}

\newcommand{\Val}{\mathrm{Val}}
\newcommand{\DivVal}{\mathrm{DivVal}}

\newcommand{\Spec}{\mathrm{Spec}}

\newcommand{\R}{\mathbb{R}}
\newcommand{\la}{\lambda}

\newcommand{\Div}{\mathrm{Div}}
\newcommand{\CDiv}{\mathrm{CDiv}}

\newcommand{\Fil}{\mathrm{Fil}}

\begin{document}
	
    \title{A note on b-divisors and filtrations on a local ring}% and local volumes of singularities}
	
    %\date{\today}

    \author{Lu Qi}%\footnote{This work was partially supported by the NSF FRG grant DMS-2139613, the NSF grant DMS-2201349 of Chenyang Xu and a Shanghai Sailing program 24YF2709800.}
    \address{School of Mathematical Sciences, East China Normal University, Shanghai 200241, China}

    \email{lqi@math.ecnu.edu.cn}

    \begin{abstract}
        In this note, we prove a correspondence between filtrations and b-divisors over a general class of Noetherian local domains. 
        As an application in the global setting, we prove a recent conjecture of Ro\'e-Urbinati.
    \end{abstract}

    \maketitle

\section{Introduction}

In this note, we consider two types of objects on normal, excellent, separated, and integral Noetherian schemes, unless otherwise specified. 
In particular, in the affine local case, denote by $(R,\fm)$ a normal, excellent Noetherian local domain.  % of pure characteristic; for example, the local ring of a normal variety over a field $k$. 
%\footnote{The condition seems necessary, since the proof relies essentially on Izumi's inequality.}
%Denote by $K\coloneqq\Frac(R)$ its fraction field and $\kappa\coloneqq R/\fm$ its residue field.

\medskip

The first kind of objects is Shokurov's \emph{b-divisors}, introduced in \cite{Sho-b-div}, where \emph{b} stands for \emph{birational}. 
A \emph{Weil b-divisor} $W$ over a scheme $X$ is a family of Weil divisors $\{W_\pi\in\Div(X_\pi)\}$ that is compatible with respect to push-forwards, where each $\pi:X_\pi\to X$ is a proper birational model of $X$. 
A Weil b-divisor $C$ over $X$ is said to be \emph{Cartier} if it is determined on some model $X_\pi$; that is, for any model $\pi':X_{\pi'}\to X$ that factors through $\pi$, $C_{\pi'}$ is the pull-back of $C_\pi$ (and, recall that $C$ is defined as the push-forward from a higher model for any other $\pi'$). 
Here, all divisors are allowed to have real coefficients.

If $X$ is a normal variety over a field $k$ of characteristic $0$, then b-divisors are also closely related to its \emph{Zariski-Riemann space} $\fX=\varprojlim_\pi X_\pi$, which dates back to Zariski in his proof of the resolution of singularities in dimensions $2$ and $3$. 
In this case, the space of Weil (resp. Cartier) b-divisors over $X$ is the same as that of Weil (resp. Cartier) divisors on $\fX$. Therefore, we will denote the space by $\Div(\fX)$ (resp. $\CDiv(\fX)$) for an arbitrary $X$\footnote{Although we will not touch the precise definition of Riemann-Zariski spaces in the generality of this note, the notation of \cite{BdFF-vol-iso} seems to be illuminating and is therefore adopted here.}.  

We remark that Cartier b-divisors appear naturally in birational geometry, for example, as the moduli part in the canonical bundle formula \cite{Kaw-subadj2}, or the nef part of a \emph{generalized polarized pair} introduced in \cite{BZ-gen-pair}. 
More recently, the K-stability theory is also extended to the setting of b-divisors \cite{DR-birational-K}. 

\medskip

The other kind of objects, \emph{graded sequences of ideals}, introduced in \cite{ELS-graded}, is defined as a decreasing sequence of ideal sheaves $\{\fa_m\subset \cO_X\}_{m\in\bZ_{\ge 0}}$ such that $\fa_m\cdot\fa_n\subset \fa_{m+n}$. 
Graded sequences have been exploited in commutative algebra and the study of singularities, including \cites{ELS-uniform-approx,Mus-regmult,LM-Okounkov-body,dFH-normal-sing,Cut13,Cut-adv,Dat-uniform-p}.
Moreover, they provide a bridge between the algebraic theory and the analytic theory of singularities. For example, using this notion, an algebraic approach toward the (strong) openness conjecture of Demailly-Koll\'ar \cite{DK-openness} was proposed in \cite{JM-val-ideal-seq}; the algebraic proof of the openness conjecture was completed by Xu \cite{Xu-quasi-monomial}. See also \cites{BFJ-regular-PSH,BFJ-singular-metric} for more applications to (complex and non-archimedean) analytic geometry.

The notion is not only a key technical tool in the local K-stability theory but also a central object itself; see, for example, \cites{Blu-minimizer-exist,LX-stability-kc,Liu-vol-sing-KE,XZ-sdsing}. 
Recently, the idea of viewing the space of all graded sequences (or their continuous version, \emph{filtrations}, where the indices $\lambda$ are real numbers) as a geometric object was implicitly employed in \cite{XZ-minimizer-unique} and then formalized in \cite{BLQ-convexity}. 
In particular, there are several topologies, including a natural geodesic metric space structure analogous to the Darvas metric on the space of K\"ahler potentials \cite{Dar-metric} (see also \cite{BJ-global-metric} for the non-archimedean version); see \cite{Qi-local} for more details and a few more results on the structure of the space of filtrations. 

\medskip

It is known that b-divisors and filtrations are naturally related. Indeed, an ideal $\fa\subset \cO_X$ defines an $X$-nef Cartier b-divisor (determined on the normalized blow-up of $\fa$)
\[
    Z_X(\fa)\coloneqq -\sum_{P/X} \ord_P(\fa)\cdot P,
\]
where the sum is taken over all prime divisors $P$ over $X$. 
It was observed in \cite{BdFF-vol-iso} that the same construction works for filtrations; that is, one has a well-defined b-divisor
\begin{equation}\label{eqn:filtr to b-div}
    Z_X(\fa_\bullet)\coloneqq -\sum_{P/X} \ord_P(\fa_\bullet)\cdot P,
\end{equation}
which is not necessarily Cartier.

Conversely, a b-divisor $W$ over $X$ defines an ideal $\fa(W)\subset\cO_X$ (that is not coherent in general) by
\begin{equation}\label{eqn:b-div to filtr}
    \fa(W)(U)\coloneqq \{f\in \cO_X(U)\mid Z(f)\le W\}
\end{equation}
for any open subset $U\subset X$. Here, we write $Z\ge 0$ for a Weil b-divisor $Z$ if all of its coefficients are non-negative. 
%Now, the sequence $\{\fa(m W)\}_{m\in\bZ_{>0}}$ automatically satisfies the condition that $\fa(m W)\cdot \fa(m' W)\subset \fa((m+m')W)$ as $Z(fg)=Z(f)+Z(g)$. 
If $W\le 0$, then $\fa(\lambda' W)\subset \fa(\lambda W)$ for any $\lambda'>\lambda$, and only then is the family of ideals decreasing. 
Thus, we will focus on anti-effective b-divisors in the sequel.

\medskip

Now we restrict ourselves to the (affine) local setting, where $X=\Spec R$ for some normal excellent Noetherian local domain $(R,\fm)$, and $x\in X$ corresponds to the closed point $\fm$. 
A b-divisor $W$ is said to be \emph{over $x\in X$} if all prime divisors appearing in all the $W_\pi$'s are centered on $x\in X$. 
Denote by $\Div(\fX,x)\subset \Div(\fX)$ the space of all such b-divisors.
Similarly, we have the space of all \emph{$\fm$-filtrations} $\fa_\bullet$ of $R$, where each $\fa_\lambda$ is an $\fm$-primary ideal. 
In particular, there is a canonical anti-effective b-divisor $Z(\fm)\in\Div(\fX,x)$ associated to the local ring $(R,\fm)$, and a canonical $\fm$-filtration $\fm^\bullet\coloneqq\{\fm^{\lceil \lambda \rceil}\}_{\lambda\in\bR_{> 0}}$, and we know that $Z(\fm^\bullet)=Z(\fm)$.  

The purpose of this note is to establish a correspondence between b-divisors and filtrations. 
To this end, we will impose some extra conditions that arise naturally.
\begin{itemize}
    \item By definition, for any $\fm$-filtrations $\fa_\bullet$, there exists $C\in\bR_{>0}$ such that $\fm^{\lceil C\lambda \rceil}\subset \fa_\lambda$. Therefore, a b-divisor $W$ coming from an $\fm$-filtration via \eqref{eqn:filtr to b-div} should automatically satisfy the inequality $C\cdot Z(\fm)\le W$. 
    Denote the set of anti-effective b-divisors bounded from below by
    \[
        \Div^+(\fX,x)\coloneqq\{W\in\Div(\fX,x)\mid C\cdot Z(\fm)\le W\le 0~ \text{for some}~ 0<C\}.
    \]

    \item An $\fm$-filtration $\fa_\bullet$ is \emph{linearly bounded} (by $\fm^\bullet$) if there exists a constant $\epsilon\in\bR_{>0}$ such that $\fa_\bullet\subset \fm^{\lceil \epsilon \lambda\rceil}$. Correspondingly, denote the set of \emph{bounded} b-divisors by
    \[
        \Div^b(\fX,x)\coloneqq\{W\in\Div(\fX,x)\mid C\cdot Z(\fm)\le W\le \epsilon\cdot Z(\fm)~ \text{for some}~ 0<\epsilon<C\}.
    \]
    
    \item We can rewrite \eqref{eqn:b-div to filtr} as
    \[
        \fa(W)(X)=\{f\in R\mid \ord_P(f)\ge -\ord_P(W)~ \text{for any divisor}~ P~\text{over}~ x\in X\}\subset R,
    \]
    which is similar to the definition of a \emph{saturated} filtration (see Definition \ref{defn:saturation}). 
    Write $\Fil^s_{R,\fm}$ for the set of linearly bounded saturated filtrations. 
\end{itemize}

The first two points will be summarized again in Lemma \ref{lem:bounded below gives filtration}. 

\medskip

Now we are ready to state the main result, which is a correspondence between b-divisors and filtrations in the local setting. 

\begin{theorem}\label{thm:structure of fil and nef}
	Let $(R,\fm)$ be a normal, excellent Noetherian local domain. Write $x\in X$ for the closed point $\fm\in \Spec R$. Then the following statements hold.
	
	\begin{enumerate}
        \item The formula \eqref{eqn:filtr to b-div} gives an injective map
        \[
            Z_X(\bullet):\Fil^s_{R,\fm}\to\Div^b(\fX,x),
        \]
        which is continuous in the coefficientwise topology. %and the closure of the image is $\Nef_+(\fX,x)$, 
		
        \item The global sections $\fa_\lambda(W)\coloneqq \fa(\lambda W)(X)$ defined as in \eqref{eqn:b-div to filtr}
        for $\lambda\in\bR_{>0}$ give a surjective map
        \[
            \fa_\bullet(\bullet):\Div^b(\fX,x)\to\Fil^s_{R,\fm}.
        \]
		
        \item We have 
        \[
            \fa_\bullet(Z_X(\fa_\bullet))=\fa_\bullet
        \]
        for any $\fa_\bullet\in\Fil^s_{R,\fm}$, and 
        \[
            Z_X(\fa_\bullet(W))\le W
        \]
        for any $W\in \Div^b(\fX,x)$.
	\end{enumerate}
    
\end{theorem}

The proof of Theorem \ref{thm:structure of fil and nef} is not hard once we have the correct set-up. The major motivation of this note is to prove the following conjecture of Ro\'e-Urbinati, \cite[Conjecture 3.19]{RU-b-div}, regarding valuations on projective varieties, which is known when $\dim X=2$ (Corollary 6.5 of \emph{op. cit.}). 
    
\begin{conjecture}\label{conj:RU}
    Let $X$ be a smooth projective variety over an algebraically closed field $k$ of characteristic $0$, and let $v\in \Val_X$ be a valuation on $X$. 
    Then $v$ is b-divisorial if and only if $Z_X(\fa_\bullet(v))\ne 0$. 
\end{conjecture}

Here, roughly speaking, a valuation $v\in \Val_X$ is called \emph{b-divisorial} following \cite{RU-b-div} if 
\[
    \fa_{\bullet}(v)=\fa_\bullet(Z_X(\fa_\bullet(v))),
\]
which follows from Theorem \ref{thm:structure of fil and nef} in the local setting where $X=\Spec R$. See Section \ref{ssec:filtr and b-div} for the precise definition and a brief discussion.

As an application of the main theorem, we prove that the statement in Conjecture \ref{conj:RU} holds in a more general setting.

\begin{corollary}\label{cor:b-div filtration}
    Let $X$ be a normal projective variety over a field $k$, and let  $v\in \Val_X$ be a valuation on $X$. Then $v$ is b-divisorial if and only if $Z_X(\fa_\bullet(v))\ne 0$.
\end{corollary}

    %The subscript $Z_X$ indicates that the b-divisor is computed over $X$; in other words, the sum in \eqref{eqn:b-divisor} is taken over all prime divisors over $X$. 

\medskip 

\begin{remark}
\begin{enumerate}
    \item One key ingredient in the proof of Theorem \ref{thm:structure of fil and nef} is that a linearly bounded valuation is saturated (Lemma \ref{lem:linearly bounded implies saturated}), which follows from Rees' theorem for filtrations \cite[Theorem 1.4]{BLQ-convexity}.
    In particular, although the volume (or multiplicity) of a filtration is not involved in any statements of the results, it is implicitly used in the proof. 

    %We also remark that the proof of Corollary \ref{cor:b-div filtration}
    
    \item The assumption of normality and excellence is twofold. 
    First, a normal excellent Noetherian local domain is analytically irreducible by \cite[\href{https://stacks.math.columbia.edu/tag/0C23}{Lemma 0C23}]{stacks-project}, and hence one can apply the results of \cite{BLQ-convexity}. 
    Second, an excellent ring is Nagata by \cite[\href{https://stacks.math.columbia.edu/tag/07QV}{Lemma 07QV}]{stacks-project}, and it is more convenient to describe the set of all \emph{proper} birational models of an excellent scheme; see Remark \ref{rmk:excellence}.
    In addition, it is natural to work with Weil divisors on normal schemes.
    
    The main theorem might hold in more general settings of a more algebraic nature. For example, the recent work of Cutkosky \cite{Cut-unramified} extended Rees' theorem and several similar results for filtrations to the setting of analytically unramified Noetherian domains, which is the same as the setting for Rees' original version \cite{Rees-Rees} for ideals.
    We will not pursue utmost generality in this note and will leave the possible generalizations to interested readers.

    \item Since pull-backs always preserve numerical equivalence, it is reasonable to say that a Cartier b-divisor $C\in\CDiv(\fX)$ is \emph{$X$-nef} if any determination $C_\pi$ is nef over $X$. 
    To define numerical classes of Weil b-divisors, it is better to consider only smooth varieties (or at least, $\bQ$-factorial varieties).
    In particular, 
    \[
        N_1(\fX)\coloneqq \varprojlim_\pi N_1(X_\pi)
    \]
    is well-defined if the subset of smooth models of $X$ is co-final with the set of all normal models of $X$. This is the case if we know the (suitable version of) resolution of singularities; for example, it is true if $X$ is a variety over a field of characteristic $0$, or if $\dim X=2$. 
    In such cases, we know that $Z_X(\fa_\bullet)$ is a \emph{$X$-nef} Weil b-divisor by the same argument as in \cite[Proposition 2.1]{BdFF-vol-iso}. 
    Positivity of a general Weil b-divisor seems to be a subtle topic, and is beyond the scope of this note.% Meanwhile, 

    \item In line with \cite[Section 2.4]{BdFF-vol-iso} and \cite[Example 4.7]{RU-b-div}, the author learned from communications with Wenbin Luo that the last statement in Theorem \ref{thm:structure of fil and nef} might be related to the b-divisorial version of Zariski-Nakayama decompositions in the local setting. Such decompositions in the global setting have been studied in \cites{BFJ-differentiability,KM-b-Zariski}. 
\end{enumerate}
\end{remark}

\subsection*{Acknowledgments}
The author would like to thank Chenyang Xu for bringing the conjecture to his attention, as well as for valuable discussions. He would like to thank Yujie Luo, Wenbin Luo, and Tong Zhang for their helpful comments.

This work is partially supported by the NSFC (No. 12501055) and the Shanghai Sailing Program (24YF2709800).

\section{Preliminaries}

In this section, we collect the definitions and basic results needed for the proof as concisely as possible.  
All schemes are assumed to be normal, excellent, separated, integral, and Noetherian unless otherwise specified. 
In particular, denote by $(R,\fm)$ an $n$-dimensional normal excellent Noetherian local domain. 
Denote by $K\coloneqq\Frac(R)$ its fraction field and $\kappa\coloneqq R/\fm$ its residue field.
%We also assume that $R$ contains a field $k$; that is, $R$ is of pure characteristic.

\subsection{Norms and valuations}

%\begin{definition}\label{defn:norm}
	A (real) \emph{norm}\footnote{We should emphasize that, in this note, a \emph{norm} is assumed to be \emph{sub-multiplicative}. This is slightly different from the typical term in Berkovich analytic geometry, where all norms are \emph{multiplicative}.} on $K$ is a function $\chi:K^\times\to \bR$ such that 
	\begin{enumerate}
		%\item $\chi(a)=0$ for any $a\in k^\times$, 
			
		\item $\chi(f+g)\ge \min\{\chi(f),\chi(g)\}$ for any $f,g\in R$, and
		
		\item $\chi(fg)\ge \chi(f)+\chi(g)$ for any $f,g\in R$. 
    \end{enumerate} 

    Set $\chi(0)\coloneqq +\infty$ by convention.
    A norm $\chi$ is called an \emph{$\fm$-norm on $R$} if it further satisfies
    \begin{enumerate}
        \item[(3)] $\chi(f)\ge 0$ for any $f\in R$ and $\chi(f)>0$ if and only if $f\in\fm$. 
    \end{enumerate}
	
    If $\chi$ is a $\fm$-norm on $R$, and we replace condition (2) above with
    \begin{enumerate}
        \item[(2')] $\chi(fg)=\chi(f)+\chi(g)$ for any $f,g\in R$,
    \end{enumerate}
    then $\chi$ is called a (real) \emph{valuation} on $R$ centered at $\fm$. 
    Denote the set of real valuations on $R$ centered at $\fm$ by $\Val_{R,\fm}$. 

%\end{definition}

A valuation $v$ on $K$ induces a \emph{valuation ring} $\cO_{v} : = \{ f \in K \, \vert \, v(f) \geq 0\}$, which is a local ring.
We write $\fm_v$ for the maximal ideal of $\cO_v$ and $\kappa_{v}:=R_v/\fm_v$.

If $R$ is a $k$-algebra, then in all definitions above, we further assume that $\chi$ is a $k$-norm; that is, $\chi(a)=0$ for any $a\in k^\times$. 
%Note that $v$ has center on $\fm$

\medskip

In general, given an integral scheme $X$, a valuation on its function field $K(X)$ is called a valuation \emph{on $X$} if there exists a scheme-theoretic point $\zeta\in X$ such that the local ring $(\cO_v,\fm_v)$ \emph{dominates} the local ring $(\cO_{X,\zeta},\fm_\zeta)$; that is, $\cO_{X,\zeta}\subset \cO_v$ and $\fm_\zeta=\fm_v\cap \cO_{X,\zeta}$.  
In this case, $\zeta\coloneqq c_X(v)$ is called a \emph{center} of $v$ on $X$. 

If $X$ is separated, then \emph{the} center of $v$ on $X$ is unique; if $X$ is proper, then any valuation on $K$ admits a unique center on $X$.
Denote by $\Val_X$ the set of valuations on $X$, and by $\Val_{X,\zeta}$ the set of valuations with center $\zeta\in X$. 

\subsubsection{Divisorial valuations}
Let $X$ be an integral separated Noetherian scheme. A valuation $v\in\Val_X$ is \emph{divisorial} if 
\[
    {\rm tr.deg}_{\kappa(\zeta)}(\kappa_v) =\dim \cO_{X,\zeta}-1,
\]
where $\zeta=c_X(v)\in X$. We write $\DivVal_X\subset \Val_X$ for the set of such valuations.
In the local setting, $v\in \Val_{R,\fm}$ is divisorial if 
${\rm tr.deg}_{\kappa}(\kappa_v) =n-1$, and we have $\DivVal_{R,\fm}\subset \Val_{R,\fm}$. 
By definition, $v\in\Val_X$ is divisorial if and only if it is divisorial on $\cO_{X,c_X(v)}$.
See \cite[Section 9.3]{HS-integral-closure} for an account of divisorial valuations from a more algebraic perspective.

Divisorial valuations appear geometrically.  
If $\mu:Y\to \Spec R$ is a proper birational morphism with $Y$ normal and  $E\subset Y$ a prime divisor, 
then there is an induced valuation  $\ord_{E} : K^\times \to \bZ$. 
If $\mu(E) =\fm$ and  $c\in \bR_{>0}$, then $c\cdot \ord_E \in \DivVal_{R,\fm}$.
As $R$ is excellent, all divisorial valuations are indeed of this form; see, for example, \cite[Lemma 6.5]{CS-arbitrary}.

%In particular, for any normal variety $X$ over a field $k$, the set $\DivVal_X$ can be identified with the set of prime divisors over $X$.

\subsection{Filtrations}
Recall that an \emph{$\fm$-filtration} on $R$ is a slight generalization of a graded sequence of $\fm$-primary ideals studied in \cite{JM-val-ideal-seq}, which is also a local analog of a filtration of the section ring of a polarized variety in \cite{BHJ-DH-measure}.

%\begin{definition}\label{defn:filtration}
An \emph{$\fm$-filtration} on $R$ is a collection $\fa_\bullet =\{\fa_\la \}_{\la \in \R_{>0}}$  of $\fm$-primary ideals of $R$ such that
\begin{enumerate}
\item (decreasing) $\fa_\la \subset \fa_{\mu}$ when  $\la >\mu$,
\item (left continuous) $\fa_{\la} = \fa_{\la-\epsilon}$ when $0<\epsilon\ll1$, and
\item (multiplicative) $\fa_{\la} \cdot \fa_{\mu} \subset \fa_{\la+ \mu}$ for any $\lambda,\mu\in\bR_{>0}$.
\end{enumerate}
By convention, we set $\fa_{0}\coloneqq R$.  

An $\fm$-filtration $\fa_\bullet$ is \emph{linearly bounded} if there exists $c\in\bR_{>0}$ such that $\fa_{\la} \subset \fm^{\lceil c  \la \rceil }$ for all $\la\in \bR_{>0}$. 
Denote the set of linearly bounded $\fm$-filtrations on $R$ by $\Fil_{R,\fm}$. 
%\end{definition}

For $v\in \Val_{R,\fm}$ and an ideal $\fa\subset R$, set
$v(\fa)\coloneqq \min\{v(f)\mid f\in\fa\}$.
For an $\fm$-filtration $\fa_\bullet$,  set
\[
v(\fa_\bullet)\coloneqq  \lim_{ \bN\in m \to \infty} \frac{ v(\fa_m)}{m} = \inf_{m\in \bN} \frac{v(\fa_m)}{m} , 
\]
where the existence of the limit and the second equality is \cite[Lemma 2.3]{JM-val-ideal-seq}.

\medskip

In general, let $X$ be a scheme. A family $\{\fa_\lambda\subset\cO_X\}_{\lambda\in\bR_{>0}}$ of (coherent) ideals is a (coherent) filtration on $X$ if it is decreasing, left continuous, and multiplicative. 
A valuation $v\in\Val_X$ can be evaluated on a coherent ideal $\fa\subset \cO_X$ by
\[
    v(\fa)\coloneqq\min\{v(f)\mid f\in\fa_{c_X(v)}\}=v(\fa_{c_X(v)}),
\]
and thus $v(\fa_\bullet)$ can be defined asymptotically as above. 

A special case that is particularly useful in this note is the \emph{valuative ideals} $\fa_\bullet(v)$ for a valuation $v\in\Val_X$ with $c_X(v)=\zeta$, defined as
\[
    \fa_\lambda(v)(U)\coloneqq\left\{\begin{aligned}
        &\{f\in \cO_X(U)\mid v(f)\ge \lambda\}, &\text{if}~ \zeta\in U\\
        &\cO_X(U), &\text{otherwise}.
    \end{aligned}\right.
\]
for $\lambda\in\bR_{>0}$. It is easy to see that $\{\fa_\lambda(v)_\zeta\}$ is an $\fm_\zeta$-filtration on $\cO_{X,\zeta}$. 
Moreover, we know that $\fa_\lambda(v)_\xi\subset\fm_{\xi}$ for any $\xi\in\overline{\zeta}$, and that $\fa_\lambda(v)_\xi=\cO_{X,\xi}$ otherwise. 
Therefore, $w(\fa_\bullet(v))=0$ for any $w\in\Val_X$ with $c_X(w)\in\overline\zeta$. 

\medskip

As in the global setting, e.g., \cites{BHJ-DH-measure,BJ-global-metric}, there is a correspondence between $\fm$-filtrations and $\fm$-norms. 
For the following version, see \cite[Definition-Lemma 2.8]{Qi-local}.
%See \cites{BHJ-DH-measure,BJ-global-metric} for a general dictionary.%, and see \cite{Qi-local} for some refinements in the local setting.

\begin{lemma}\label{lem:filtr and norm}
    There is a one-to-one correspondence between $\fm$-norms and $\fm$-filtrations satisfying $\cap_{\lambda>0}\fa_\lambda=\{0\}$, as follows.
    
    Given an $\fm$-filtration $\fa_\bullet$, the associated $\fm$-norm $\ord_{\fa_\bullet}$ is defined by
    \[
        \ord_{\fa_\bullet}(f)\coloneqq\sup\{\lambda\in\bR\mid f\in\fa_\lambda\}.
    \]
    Conversely, given an $\fm$-norm $\chi$, the associated $\fm$-filtration $\fa_\bullet=\fa_\bullet(\chi)$ is defined by 
    \[
        \fa_\lambda(\chi)\coloneqq\{f\in R\mid \chi(f)\ge \lambda\}
    \]
    for any $\lambda\in\bR_{>0}$. \qed
\end{lemma}

An $\fm$-norm is called \emph{linearly bounded} if its associated $\fm$-filtration is so.

\subsubsection{Saturated filtrations}

The following notion introduced in \cite{BLQ-convexity} plays a key role in our proof. 

\begin{definition}\label{defn:saturation}
    The \emph{saturation} $\widetilde{\fa}_\bullet$ of an $\fm$-filtration $\fa_\bullet$ is defined by
    \[
        \widetilde \fa_\lambda\coloneqq \{f\in \fm \mid v(f)\ge \lambda\cdot v(\fa_\bullet) \text{ for all } v\in\DivVal_{R,\fm} \}
    \]
    for each $\lambda\in \R_{>0}$.
    We say that $\fa_\bullet$ is \emph{saturated} if $\fa_\bullet=\widetilde\fa_\bullet$. 
\end{definition}

Denote the set of linearly bounded saturated filtrations by 
\[
    \Fil^s_{R,\fm}\coloneqq \{\fa_\bullet\in\Fil_{R,\fm}\mid \fa_\bullet \text{ is saturated}\}.
\]

Saturation is analogous to the integral closure of an ideal. 
In particular, it provides the right notion to characterize when several inequalities of multiplicities are indeed equalities; for more details, see \cite{BLQ-convexity} or \cite{Cut-unramified} for a more general account.

\subsection{Shokurov's b-divisors}

In this subsection, we collect some basic results about Shokurov's b-divisors, mostly to fix the notation. Recall that all schemes involved are normal, excellent, integral, separated, and Noetherian. 
In particular, we can talk about the group of Weil divisors $\Div(X)$, and there is an injective cycle map $\CDiv(X)\hookrightarrow \Div(X)$ from the group of Cartier divisors.  
For more details in the geometric setting, see \cites{Isk-b-div,Cor-flip,BdFF-vol-iso}. 

\begin{comment}
Recall that the \emph{Riemann-Zariski space} $\fX$ of a scheme $X$ is defined to be the inverse limit
    \[
        \fX:=\varprojlim_{\pi} X_\pi,
    \]
where the limit is taken in the category of locally ringed topological spaces, and over all proper birational morphisms $\pi:X_\pi\to X$, where $\pi'\ge \pi$ if and only if $\pi'$ factors through $\pi$. Here each $X_\pi$ is called a \emph{model} of $X$. Following Shokurov, one can define b-divisors on $X$ using $\fX$ as follows.
\end{comment}
A \emph{(proper birational) model} of a scheme $X$ is a proper birational morphism $\pi:X_\pi\to X$. The set of all models of $X$ forms a directed set ordered by \emph{dominance}. In other words, define a partial order by $\pi'\ge \pi$ if and only if $\pi'$ factors through $\pi$, and for any models $\pi_1,\pi_2$ of $X$, there is a model $\pi'$ such that $\pi'\ge \pi_i$ for $i=1,2$. 

\begin{remark}\label{rmk:excellence}
    Such a higher model $X_{\pi'}$ can be taken as the normalization of the fiber product $X_{\pi_1}\times_X X_{\pi_2}$. The assumption that $X$ is excellent is used to guarantee that the normalization is finite; see \cite[\href{https://stacks.math.columbia.edu/tag/07QV}{Lemma 07QV}]{stacks-project}.
\end{remark}

Recall that an integral \emph{Weil b-divisor} over $X$ is an element of the abelian group
$$\varprojlim_\pi \Div(X_\pi),$$
where the limit is taken with respect to the push-forward maps $\Div(X_{\pi'})\to\Div(X_\pi)$ for $\pi'\ge \pi$. 
More concretely, a Weil b-divisor $W$ over $X$ is a family of Weil divisors $\{W_\pi\in\Div(X_\pi)\}$ for all (proper, normal, birational) models of $X$, such that $\mu_*W_{\pi'}=W_\pi$ if $\pi'$ factors through $\mu:X_{\pi'}\to X_\pi$. 

Similarly, an integral \emph{Cartier b-divisor} over $X$ is an element of the abelian group
$$\varinjlim_\pi \CDiv(X_\pi),$$
where the limit is taken with respect to the pull-back maps $\CDiv(X_{\pi})\to\CDiv(X_{\pi'})$ for $\pi'\ge \pi$. 
Equivalently, a Cartier b-divisor $C$ over $X$ is a Weil b-divisor $\{C_\pi\}$ \emph{determined} on a model $X_\pi$; that is, $C_{\pi'}=\mu^*C_\pi$ for any $\pi'$ factoring through $\mu:X_{\pi'}\to X_\pi$. 
	
Denote by 
\[
\Div(\fX):= \varprojlim_\pi\Div(X_\pi)\otimes\bR\quad (\text{resp.}~ \CDiv(\fX):=\varinjlim_\pi\CDiv(X_\pi)\otimes\bR)
\]
the real vector space of $\bR$-Weil b-divisors (resp. $\bR$-Cartier b-divisors). 
Note that there is a cycle map $\CDiv(\fX)\hookrightarrow\Div(\fX)$ induced by those on models.

%Slightly abusing the notation, we will call them $\bR$-b-divisors (resp. $\bR$-Cartier b-divisors). We equip $\Div_\bR(\fX)$ with the topology of coefficient-wise convergence.
	
Recall that in our setting, any divisorial valuation $v$ is of the form $t\cdot \ord_E$, where $E\subset Y\to X$ is a Weil divisor on a model of $X$. Hence for any $W\in\Div(\fX)$, we can define $v(W):=t\cdot\ord_E(W_Y)$. 
As the name suggests, a net $\{W_m\}\subset \Div(\fX)$ converges to $W$ in the \emph{coefficientwise topology} if and only if $\ord_P(W_m)\to \ord_P(W)$ for any prime divisor $P$ over $X$. 
	
\begin{lemma}(c.f. \cite[Lemma 1.1]{BdFF-vol-iso}\label{lem:b-divisors as functions})
	There is an injection from $\Div(\fX)$ into homogeneous functions $\DivVal_X\to \bR$, defined by $W\mapsto (\phi_W:v\mapsto v(W))$. The image consists of those $\phi$ such that, for any model $X_\pi$, the set of prime divisors $E\subset X_\pi$ such that $\phi(E)\ne 0$ is finite. 
\end{lemma}
	
\begin{example}\label{eg:b-div of an ideal}
    Let $X$ be a scheme.
    \begin{enumerate}
        \item A Cartier divisor $D\in\CDiv(X)$ defines a Cartier b-divisor $\overline D$ determined on $X$, where for any model $\pi$, $(\overline D)_\pi\coloneqq \pi^*D$. 
        More generally, a Cartier divisor on a model defines a Cartier b-divisor in a similar manner. 

        \item Given a coherent ideal $\fa\subset \cO_X$, one can define $Z_X(\fa)$, the Cartier b-divisor determined on the normalized blowup $X_\pi$ of $X$ along $\fa$, to be
        \[
		      \fa\cdot\cO_{X_\pi}=\cO_{X_\pi}(Z_X(\fa)_\pi).
        \]
        Equivalently, $Z_X(\fa)=-\sum_P \ord_P(\fa)\cdot P$ is the b-divisor corresponding to the function 
        \[
            \phi_\fa:\DivVal_X\to \bR,~ v\mapsto -v(\fa).
        \]
        Similarly, rational function $f\in K(X)$ defines a Cartier b-divisor 
        \[
            Z_X(f)\coloneqq -\sum_P \ord_P(f)\cdot P=-\overline{\mathrm{div}(f)}.
        \] 
    \end{enumerate}
\end{example}

\subsubsection{b-divisors over a closed point}

Let $x\in X$ be a closed point. A b-divisor $W$ is \emph{over} $x$ if for any model $X_\pi$, every component $P$ of the $\bR$-Weil divisor $W_\pi$ is centered on $x\in X$; that is, $\pi(P)=\{x\}$. 
Denote by $\Div(\fX,x)\subset \Div(X)$ the subspace of all b-divisors over $x$. It is easy to see that $Z(\fa)\in\Div(\fX,x)$ for any $\fm$-primary ideal $\fa$.  

Denote the set of anti-effective b-divisors bounded from below by 
\[
    \Div^+(\fX,x)\coloneqq \{Z\in\Div(\fX,x)\mid C\cdot Z(\fm)\le Z\le 0 \text{ for some } C\in\bR_{>0}\},
\]
and denote the set of \emph{bounded} (anti-effective) b-divisors by
\[
    \Div^b(\fX,x)\coloneqq \{Z\in\Div(\fX,x)\mid C\cdot Z(\fm)\le Z\le \epsilon\cdot Z(\fm) \text{ for some } \epsilon,C\in\bR_{>0}\}.
\]

\subsection{Filtrations and b-divisors}\label{ssec:filtr and b-div}

By Fekete's lemma, the same construction as in Example \ref{eg:b-div of an ideal} works for filtrations. More precisely, we have the following construction. 

\begin{lemma}(c.f. \cite[Lemma 2.11]{BdFF-vol-iso} or \cite[Proposition 3.1]{RU-b-div})
    Given a coherent filtration $\fa_\bullet$ on a scheme $X$, one can define an $\bR$-Weil b-divisor 
    \begin{equation*}%\label{eqn:b-divisor}
        Z_X(\fa_\bullet)\coloneqq -\sum_P \ord_P(\fa_\bullet)\cdot P,
    \end{equation*}
    where the sum is taken over all prime divisors $P$ over $X$.

    If $X=\Spec R$ for some Noetherian local domain $(R,\fm)$ and $\fa_\bullet$ is an $\fm$-filtration on $R$, then $Z_R(\fa_\bullet)\coloneqq Z_X(\fa_\bullet)\in\Div(\fX,x)$. 
\end{lemma}

As in \cite{Qi-local}, the \emph{coefficientwise topology} on a space of filtrations is induced by the coefficientwise topology on $\Div(\fX)$ via $Z_X$.

\medskip
    
Conversely, if $X=\Spec R$ for some Noetherian local domain $(R,\fm)$, then for a b-divisor $W\in \Div(\fX,x)$ and $\lambda\in\bR_{>0}$, one can define an ideal of $R$ by
\begin{equation}\label{eqn:b-div to filtr local}
    \fa_\lambda(W):=\{f\in R\mid Z(f)\le \lambda W\}=\{v(f)\ge -\lambda \cdot v(W) \text{ for any } v\in\DivVal_{R,\fm}\}.
\end{equation}
This coincides with the global section $\cO_X(\lambda W)$ of \cite{BdFF-vol-iso} if $W\le 0$. 
Note that the same definition yields a filtration $\fa_\bullet$ on a scheme $X$, but it is not coherent in general. 

%using the notation of \cite{RU-b-div}, we know that $\fa_\lambda(W)=\{f\in R\mid \ord_W(f)\ge \lambda\}$.

Alternatively, for coherent filtration $\fa_\bullet$ on $X$, one can define the \emph{vanishing order} of an effective Cartier divisor $D$ on $X$ along the b-divisor $Z_X(\fa_\bullet)$ by
\begin{equation}\label{eqn:vanishing order of b-div}
    \ord_{Z_X(\fa_\bullet)}(D)\coloneqq \sup\{t\in\bR\mid \overline D+tZ_X(\fa_\bullet)\ge 0\},
\end{equation}
where $\overline D$ is the Cartier b-divisor defined in Example \ref{eg:b-div of an ideal}. 
If $\fa_\bullet=\fa_\bullet(v)$ for some $v\in\Val_X$, then we write $\ord_{Z_X(v)}$ or even $\ord_{Z(v)}$ for simplicity. 
As in \cite[Section 3]{RU-b-div}, $v$ is called \emph{b-divisorial} if $v(D)=\ord_{Z(v)}(D)$ for any effective Cartier divisor $D$ on $X$. 

\begin{remark}\label{rmk:domain for definition of vanishing order}
    Let $X=\bP^2_{[x:y:z]}$ and let $D=\mathrm{div}(\frac{1}{x})=-H_x$, where $H_x$ is the $x$-axis. Let $v=\ord_E$, where $E$ is the exceptional divisor of the blow-up of the origin. 
    Then as a trivial observation, clearly $\overline D+tZ(v)$ is never effective; on the other hand, according to \cite[Proposition 3.17]{RU-b-div}, one should have $\ord_{Z(v)}(D)=v(D)=-1$.

    Therefore, it seems necessary to restrict the definition \eqref{eqn:vanishing order of b-div} to \emph{effective} Cartier divisors.
    Equivalently, one can define $\ord_{Z(v)}(f)$ as in \eqref{eqn:vanishing order of b-div} for $f\in \cO_{X,c_X(v)}$, where the b-divisors are now computed over a neighborhood of $c_X(v)$.  
\end{remark}

We will use the following formula for $\ord_{Z(\fa_\bullet)}$. 

\begin{lemma}\label{lem:filt rewriten}
    Let $\fa_\bullet$ be a coherent filtration on a normal excellent scheme $X$. If there exists a scheme-theoretic point $\zeta\in X$ such that $\fa_{\lambda,\xi}=\cO_{X,\xi}$ for any $\xi\notin\overline{\zeta}$. Then for any $f\in \cO_{X,\zeta}$, we have
    \[
        \ord_{Z_X(\fa_\bullet)}(f)=\inf_{w\in\DivVal_X,~c_X(w)\in\overline{\zeta}} \frac{w(f)}{w(\fa_\bullet)}.
    \]
    In particular, let $v\in\Val_X$ be a valuation. Then for any $f\in \cO_{X,c_X(v)}$, we have
    \[
        \ord_{Z(v)}(f)=\inf_{w\in\DivVal_X,~c_X(w)\in\overline{c_X(v)}} \frac{w(f)}{w(\fa_\bullet(v))}.
    \]
\end{lemma}

\begin{proof}
    By definition, we compute that
    \begin{equation*}
    \begin{aligned}
        \ord_{Z_X(\fa_\bullet)}(f)=&\sup\{t\in\bR\mid \overline{\mathrm{div}(f)}+t Z_X(\fa_\bullet)\ge 0~\text{near}~\zeta\}\\
            =&\sup\{t\in\bR\mid w(f)-t w(\fa_\bullet)\ge 0~\text{for any}~ w\in\DivVal_X~\text{with}~ c_X(w)\in\overline{\zeta}\}\\
            =&\sup\{t\in\bR\mid t\le \frac{w(f)}{w(\fa_\bullet)}~\text{for any}~ w\in\DivVal_X~\text{with}~ c_X(w)\in\overline{\zeta}\}\\
            =&\inf_{w\in\DivVal_X,~c_X(w)\in\overline{\zeta}} \frac{w(f)}{w(\fa_\bullet)},
    \end{aligned}
    \end{equation*}
    where the first equality is definitional, the second uses the assumption that $\mathrm{div}(f)\ge 0$ near $\zeta$ and that $w(\fa_\bullet)=0$ for any $w\in\Val_X$ with $c_X(w)\notin\overline{\zeta}$, and the last two are immediate.
\end{proof}

\section{Proof of the main results}

We prove Theorem \ref{thm:structure of fil and nef} and Corollary \ref{cor:b-div filtration} in this section.

\subsection{b-divisors and filtrations over a local ring}

Let $(R,\fm)$ be a normal, excellent Noetherian local domain. Write $X=\Spec R$ and $x\in X$ for the closed point corresponding to $\fm$.
As noted in the introduction, the key ingredient is the following easy observation.

\begin{lemma}\label{lem:linearly bounded implies saturated}
    If a valuation $v\in\Val_{R,\fm}$ is linearly bounded, then $\fa_\bullet(v)$ is saturated.
\end{lemma}

\begin{proof}
    This follows immediately from Corollary 3.18 and Lemma 3.20 of \cite{BLQ-convexity}, the latter of which relies on Rees' theorem for filtrations, \emph{op. cit.}, Theorem 1.4. 
\end{proof}

To prove Theorem \ref{thm:structure of fil and nef}, it remains to prove a boundedness estimate as follows.

\begin{lemma}\label{lem:bounded below gives filtration}
	For any $W\in\Div^+(\fX,x)$, the ideal $\fa_\lambda(W)=\fa(\lambda W)(X)$ is $\fm$-primary for any $\lambda\in\bR_{>0}$, and the collection $\{\fa_\lambda(W)\}_{\lambda\in\bR_{>0}}$ is an $\fm$-filtration, denoted by $\fa_\bullet(W)$. Moreover, if $W\in\Div^b(\fX,x)$, then $\fa_\bullet(W)\in\Fil^s_{R,\fm}$.
\end{lemma}

\begin{proof}
	First, assume that there exists $C\in\bR_{>0}$ such that $W\ge CZ(\fm)$. Then, for any $\lambda\in\bR_{>0}$, $f\in\fm^{\lceil\lambda C\rceil}$, and $\ord_E\in\DivVal_{R,\fm}$, we know that 
	\[
	-\ord_E(f)\le -\ord_E(\fm^{\lceil\lambda C\rceil})=-\lceil \lambda C\rceil \ord_E(\fm)\le -\lambda C\ord_E(\fm).
	\]
	So $Z(f)\le \lambda CZ(\fm)\le \lambda W$. 
    This shows that $\fm^{\lceil\lambda C\rceil}\subset \fa_\lambda(W)$; that is, $\fa_\lambda(W)$ is $\fm$-primary.
	We always have that $\fa_\lambda(W)\cdot\fa_\mu(W)\subset\fa_{\lambda+\mu}(W)$ as $Z(fg)=Z(f)+Z(g)$ by definition. 
    
    Since $W\le 0$, we know that $\fa_\mu(W)\subset \fa_\lambda(W)$ for $\mu\ge \lambda$. 
    So $\fa_\lambda(W)\subset\cap_{\epsilon>0}\fa_{\lambda-\epsilon}(W)$. 
    Conversely, if $f\in\cap_{\epsilon>0}\fa_{\lambda-\epsilon}(W)$ for $\lambda>0$, then for any $\ord_E\in\DivVal_{R,\fm}$, we know that
	\[
	   \ord_E(Z(f))\le\lim_{\epsilon\to 0} (\lambda-\epsilon)\ord_E(W)=\lambda\ord_E(W),
	\]
	which implies that $f\in\fa_\lambda(W)$. Thus $\fa_\lambda(W)=\cap_{\epsilon>0}\fa_{\lambda-\epsilon}(W)$. By the last paragraph, we know that $\fa_\lambda(W)$ is $\fm$-primary, and hence there exists $\epsilon>0$ such that $\fa_\lambda(W)=\fa_{\lambda-\epsilon}(W)$ as $R/\fa_\lambda(W)$ is Artinian. 
    This proves that $\{\fa_\lambda(W)\}_{\lambda\in\bR_{>0}}$ is an $\fm$-filtration.
	
    Now assume that $W\le \epsilon Z(\fm)$. By the same argument as above, where the inequality is in the reverse direction, we know that $\fa_\bullet(W)$ is linearly bounded. 
    Since $\fa_\bullet(W)$ is defined using divisorial valuations, it is saturated by \cite[Proposition 2.20]{Qi-local}, and the proof is finished.
\end{proof}

    %\begin{lemma}\label{lem:main}
         %$\fa_\bullet=\fa_\bullet(Z(\fa_\bullet))$ for any $\fa_\bullet\in\Fil^s$.
    %\end{lemma}

\begin{proof}[Proof of Theorem \ref{thm:structure of fil and nef}]
    By Lemma \ref{lem:bounded below gives filtration}, we know that $\fa_\lambda(W)\coloneqq\fa(\lambda W)(X)$ defines a map 
    \[
        \fa_\bullet(\bullet):\Div_b(\fX,x)\to \Fil^s_{R,\fm}.
    \]
    Moreover, in view of \eqref{eqn:filtr to b-div} and \eqref{eqn:b-div to filtr local}, the equality $\fa_\bullet=\fa_\bullet(Z_X(\fa_\bullet))$ is just a restatement of Definition \ref{defn:saturation} for saturated filtrations, from which the injectivity of $Z_X(\bullet)$ and the surjectivity of $\fa_\bullet(\bullet)$ follow.

    The continuity of $Z_X(\bullet)$ follows from the definition of the coefficientwise topology.

    Finally, for any $\lambda\in\bR_{>0}$ and $f\in\fa_\lambda(W)$, we know that $-v(f)\le \lambda v(W)$ for any $v\in\DivVal_{R,\fm}$. Hence $-v(\fa_\lambda(W))/\lambda\le v(W)$. 
    Letting $\lambda\to\infty$, we get the inequality 
    \[
        v(Z_X(\fa_\bullet(W)))=-v(\fa_\bullet(W))\le v(W)
    \]
    for any $v\in\DivVal_{R,\fm}$. The proof is finished.
\end{proof}

\begin{remark}\label{rmk:norm}
    Using the notation of \cite{RU-b-div}, the above theorem implies that in the local setting, it suffices to assume that $\xi$ is a \emph{saturated (sub-multiplicative) norm} in order to obtain the inequality $\xi=\ord_{D_\xi}$.
\end{remark}

\subsection{Application to b-divisorial filtrations}

Our strategy to prove Corollary \ref{cor:b-div filtration} is quite different from that in \cite{RU-b-div} in dimension $2$; in particular, we avoid going into the detailed structure approximating valuations. Instead, we rely on the following general version of Izumi's inequality; see, for example, \cite[Remark 1.6]{RS-Izumi}.

\begin{lemma}\label{lem:Izumi}
    Let $(R,\fm)$ be a normal excellent Noetherian domain. Then any $v\in\DivVal_{R,\fm}$ is linearly bounded.
\end{lemma}

As in Remark \ref{rmk:norm}, we prove a slightly more general statement.

\begin{lemma}\label{lem:saturated norm}
    Let $\fa_\bullet$ be a coherent filtration of ideal sheaves on a scheme $X$. If there exists a scheme-theoretic point $\zeta\in X$ such that $\fa_{\bullet,\zeta}\in\Fil^s_{R,\fm}$, where $R=\cO_{X,\zeta}$ and $\fm=\fm_\zeta$.
    Then for any $f\in R$, we have
    \[
        \ord_{\fa_{\bullet,\zeta}}(f)=\inf_{w\in\DivVal_{R,\fm},~ c_X(w)=\zeta} \frac{w(f)}{w(\fa_{\bullet,\zeta})}=\ord_{Z_R(\fa_\bullet)}(f). 
    \]
\end{lemma}

\begin{proof}
    The second equality follows from Lemma \ref{lem:filt rewriten} applied to $\Spec R$. 
    In view of Lemma \ref{lem:filtr and norm}, the first is a restatement of the equality $\fa_\bullet=\fa_\bullet(Z_R(\fa_\bullet))$ of Theorem \ref{thm:structure of fil and nef}.
\end{proof}

Note that this readily proves the corollary when $\zeta\in X$ is a closed point. In general, we get an inequality between $v$ and $\ord_{Z(v)}$, which is the only missing ingredient. 
%In general, we will prove the reverse inequality to \cite[Proposition 3.15]{RU-b-div}.

\begin{proof}[Proof of Corollary \ref{cor:b-div filtration}]
    The forward implication is easy. 
        
    Now assume that $Z_X(\fa_\bullet(v))\ne 0$.
    Let $\zeta=c_X(v)\in X$ be the center of $v$ on $X$. Let $R\coloneqq \cO_{X,\zeta}$ with the maximal ideal $\fm\coloneqq\fm_\zeta$. 
    By definition, there is $w\in\DivVal_{R,\fm}$ such that $w(\fa_\bullet(v))>0$. 
    By Lemma \ref{lem:Izumi}, we know that $\fa_{\bullet,\zeta}(v)$ is linearly bounded, hence also saturated by Lemma \ref{lem:linearly bounded implies saturated}. 
        
    For any $f\in R$, we have the following 
    \begin{align}\label{eqn:main}
        v(f)=\inf_{c_X(w)=\zeta} \frac{w(f)}{w(\fa_\bullet(v))}\ge
            \inf_{c_X(w)\in \overline{\zeta}} \frac{w(f)}{w(\fa_\bullet(v))}=
            \ord_{Z_X(v)}(f),
    \end{align}
    where both infima are taken over $w\in\DivVal_X$; the first equality follows from Lemma \ref{lem:saturated norm}, the second follows from Lemma \ref{lem:filt rewriten}, and the inequality is trivial.

    The reverse inequality is known by \cite[Proposition 3.15]{RU-b-div}. 
    The proof is finished.
\end{proof}

%As a final remark, the first two equalities in \eqref{eqn:main} indeed follow from the equality $v(\fa_\bullet(v))=1$ combined with \cite[Lemma 2.4]{JM-val-ideal-seq}. Therefore, the full strength of Theorem \ref{thm:structure of fil and nef} is not actually needed to prove the corollary. 
%However, as in Remark \ref{rmk:norm}, the theorem implies that the same statement holds for more general norms on $K(X)$. 

%\bibliographystyle{abbrv}
\bibliography{ref}

\end{document}